\documentclass[10pt,draft,reqno]{amsart}
     \makeatletter
     \def\section{\@startsection{section}{1}%
     \z@{.7\linespacing\@plus\linespacing}{.5\linespacing}%
     {\bfseries
     \centering
     }}
     \def\@secnumfont{\bfseries}
     \makeatother
   \usepackage{amssymb,amsfonts,amssymb,amsfonts}
   \usepackage[toc,page]{appendix}
\newtheorem{theorem}{Theorem}[section]
\newtheorem{lemma}[theorem]{Lemma}
\newtheorem{proposition}[theorem]{Proposition}

\theoremstyle{definition}
\newtheorem{hypothesis}[theorem]{Hypothesis}
\newtheorem{definition}[theorem]{Definition}
\newtheorem{example}[theorem]{Example}
\theoremstyle{remark}
\newtheorem{remark}[theorem]{Remark}
\numberwithin{equation}{section}
\newcommand{\cV}{\mathcal V}

\begin{document}

\title[Taylor expansions and Castell estimates]{Taylor expansions and Castell estimates for solutions of stochastic differential equations driven by rough paths}

\author{Qi Feng*}
\address{Qi Feng: Department of Mathematics, University of Southern California, Los Angeles, CA 90089-2532, USA}
\email{qif@usc.edu}
\thanks{* Corresponding author}
\urladdr{https://sites.google.com/site/qifengmath/}

\author[Xuejing Zhang]{Xuejing Zhang}

\address{Xuejing Zhang: Investment management, Lincoln financial Group, Philadelphia, PA , 19103, USA}
\email{xuejing.zhang@lfg.com}



\begin{abstract}
We study the Taylor expansion for the solutions of differential equations driven by $p$-rough paths
with $p>2$. We prove a general theorem concerning the convergence of the Taylor expansion on a nonempty interval provided that the vector fields are analytic on a ball centered at the initial point. We also derive criteria that enable us to study the rate of convergence of the Taylor expansion. Finally, as the most original part of this paper, we prove Castell expansions and tail estimates with exponential decays for the remainder terms of the solutions of the stochastic differential equations driven by continuous centered Gaussian process with finite $2D~\rho-$variation and fractional Brownian motion with Hurst parameter $H>1/4$.
\end{abstract}

\maketitle

\section{Introduction}

 The aim of this paper is to study the stochastic Taylor series of the solutions of stochastic differential equations (SDE) driven by general Gaussian process. This idea was first introduced by R. Azencott \cite{Aze} and G. Ben Arous \cite{G.Ben} when the SDE is driven by Brownian motion.  
   A convergence result for the stochastic Taylor series was established on a non-empty time interval by G. Ben Arous, using the $L_{2}$ bound of iterated integrals of Brownian motion and a Borel-Cantelli type argument (see also F. Castell \cite{Cast}). Furthermore, a Castell expansion for the solution of SDE driven by Brownian motion was first proved by F. Castell \cite{Cast} using the methods introduced by R. Azencott \cite{Aze}. This convergence result has been extended by F. Baudoin and X. Zhang \cite{FB3} to SDE driven by fractional Brownian motion (fBm) with Hurst parameter $H>1/2$, using both a pathwise deterministic approach and a probabilistic approach. A weak type of such approximation was studied by F. Baudoin and L. Coutin \cite{FB1} for SDEs driven by fBm with Hurst parameter $1/4<H<1/2$. Our goal is to extend both the convergence and the Castell expansion results for SDEs driven by general Gaussian process. Let us summarize some  related works on Taylor expansion and Castell expansion. 

\noindent(1) By using $\mathcal{C}^1$ approximation flow, I. Bailleul \cite{BAI} proved the Taylor expansion of differential equations driven by weak geometric rough paths on Banach spaces with weaker Lipschitz conditions on the vector fields. Thus, it also gives  weaker estimates for the remainder terms of the Taylor expansions. A deterministic estimate of the remainder term of a similar Castell expansion by studying flows driven by Banach space-valued weak geometric H\"older $p$-rough path is proved in \cite{BAI2}. \\
$(2)$ Y. Boutaib, L.Gyurko, T. Lyons and D. Yang \cite{BOU} proved a dimension-free estimate of rough differential equation which also gives a remainder term estimate of the Taylor expansion. Later on, by using branched rough path introduced by M. Gubinelli \cite{Gub}, H. Boedihardjo \cite{BOE} proved that the iterated integrals of branched rough path decay factorially fast (in the tree factorial sense). Later on, H. Bordihardjo, T. Lyons and D. Yang \cite{BOE2} used the method different from \cite{Lyons} to show that the remainder term of rough Taylor expansion decay factorially fast without using the neoclassical inequality. In this paper, we use the result by T. Lyons \cite{Lyons} to get our tail estimates and we have different assumptions on the vector fields.

The paper is arranged as below.  In section 2, we focus on rough differential equation (RDE) (\ref{rough differential equation}) (RDE), where the $V_i's$ are $C^{\infty}$ vector fields on $\mathbb{R}^n$ with bounded derivatives, and the driving signal $x$ is a $d$-dimensional continuous path with bounded $p$-variation, for $p>2$. The results in this section are deterministic in the rough path framework.
\begin{align}\label{rough differential equation}
y(t)=y_0+\sum^{d}_{i=0}\int^{t}_{0}V_{i}\left(y(s)\right)dx^{i}(s) .
\end{align}
We first define the Taylor expansion associated with RDE \eqref{rough differential equation}.
Under further assumption that the vector fields $V_i$'s are analytic on the set $\{y:  \|y-y_0\| \leq C\}$ for some $C>0$, we prove a general convergence result of the Taylor series for the solution $y(t)$ of RDE (\ref{rough differential equation}). More precisely, we are able to express the solution $y(t)$ of (\ref{rough differential equation}) as the sum of its Taylor expansion on a non-empty interval. We then use the estimate of iterated integrals ( see \cite{FBa}[Theorem 7.16], also \cite{Lyons}) to provide convergence criteria that enables us to express the non-empty interval in a more quantitative way and to study the rate of convergence of the Taylor series.  In section $3$, we follow the approach introduced by Friz-Victoir where we introduce the approximating sequence to fBm with $H>1/4$ and continuous centered Gaussian process with finite $2D$ $\rho-$variation and i.i.d. components. Then our results from Section $2$ is well adapted to both cases. In particular, by using the methods introduced by R. Azencott \cite{Aze} and F. Castell \cite{Cast}, we prove a Castell expansion and tail estimate for the remainder term for the solutions of SDEs driven by centered i.i.d Gaussian process with finite $2D$ $\rho-$variation and fBm with Hurst parameter $H>1/4$. In particular, the tail estimate for fBm with $H>1/2$ verifies the claim in \cite{FBh}. We leave the proofs of technical lemmas in the appendices.

\section{Taylor expansion for differential equations driven by $p$-rough paths.}

Let us introduce the following basic notion for our use in this paper. For a detailed review and study of the rough path theory, see \cite{FH15, Pfriz, Ly2} and the references therein. 

For a $d$-dimensional $p$-rough path $x(t)=(x^1(t),\cdots,x^d(t))$, we denote rough iterated integral for $x(t)$ as
\[
\int_{\triangle^{k}[0,t]}dx^{I}=\int_{0<t_1<t_2\cdots<t_k<t}dx^{i_1}(t_1)\cdots dx^{i_k}(t_k),
\]
where $\triangle^{k}[0,t]=\{(t_1,\cdots,t_k)\in[0,t]^{k}, 0\leq t_1\leq t_2\cdots\leq t_k\leq t\}$ is a partition of $[0,t]$. In general, we denote $\mathcal{D}([s,t])$ as the set of all the partitions of time interval $[s,t]$. For simplicity, we denote $X^I_t=\int_{\triangle^{k}[0,t]}dx^{I}$ and $x^0(t)=t$. We then introduce the iterated integral of order $k$ as \[
\int_{\triangle^{k}[s,t]}dx^{\otimes k}=\sum_{I\in\{1,\cdots,d\}^{k}}\left(\int_{\triangle^{k}[s,t]}dx^{I}\right)e_{i_1}\otimes \cdots \otimes e_{i_k},
\]
where $(e_1.\cdots, e_d)$ is the canonical basis of $\mathbb{R}^d$. Define the $p$-variation norm  $\left\|\int dx^{\otimes k}\right\|_{p-var, [s,t]}$  as
\[
\left\|\int dx^{\otimes k}\right\|_{p-var, [s,t]}\equiv\left( \sup_{\Pi\in \mathcal{D}[s,t]}\sum^{n-1}_{i=0}\left\|\int_{\triangle^{k}_{[t_{i}, t_{i+1}]}}dx^{\otimes k}\right\|^{p}\right)^{1/p}.
\]      

For a word $I=(i_1,\cdots,i_k)\in \{0,\cdots,d\}^k$, we call $|I|$ the size of $I$ which equals $k$ here. Then for vector fields $V_0,V_1,\cdots,V_d$  on $\mathbb{R}^n$, we denote $V_I$ as the iterated Lie brackets of the  vector fields $V_i's$, 
\[
V_I=[V_{i_1}[V_{i_2},\cdots[V_{i_{k-1}},V_{i_k}]]\cdots],\quad \text{for}\quad I=(i_1,\cdots,i_k).
\]
 We use the notation $\Lambda_I$ as below (see \cite{Cast} for details):
\begin{align}\label{lambda notation}
\Lambda_I(x)_t=\sum_{\sigma \in \sigma_{|I|}}\frac{(-1)^{e(\sigma)}}{|I|^2{|I|-1\choose e(\sigma)}}x(t)^{I\circ \sigma^{-1}},	
\end{align}

where $\sigma$ is a permutation of size $|I|$, and we denote $\sigma_{|I|}$ as the set of all permutations of size $|I|$. We denote $e(\sigma)$ ( see \cite{Str} for details) as the cardinality of the error set $\{i\in \{1,\cdots,k-1\}; \sigma(i)>\sigma(i+1)\}$. For a word $I$ of size $k$, we have $I\circ \sigma=(i_{\sigma(1)},\cdots,i_{\sigma(k)}).$

\subsection{Taylor expansion of the solution}
Throughout this section, we study the RDE in the rough path sense which is always deterministic. 
The basic equation we consider is the following,
\begin{align}
\label{3.2}
y(t)=y_0+\sum^{d}_{i=0}\int^{t}_{0}V_{i}\left(y(s)\right)dx^{i}(s) 
\end{align}
We make the following hypothesis throughout this section.
\begin{hypothesis}\label{hypo: vector-path}
\noindent\emph{$(i)$} 
The $V_i$'s are $C^{\infty}$ vector fields on $\mathbb{R}^n$ with bounded derivatives, and analytic on the set $\{y:  \|y-y_0\| \leq C\}$ for some $C>0$.\\
\noindent\emph{$(ii)$}  The driving path $x:[0,T] \rightarrow \mathbb{R}^{d}$ is geometric $p$-rough path $(p>2)$ with approximating sequence $x_n\in C^{1-var}([0,T],\mathbb{R}^{d})$ converging in the $p$-variation topology.
\end{hypothesis}
\begin{remark}\label{remark:approximation}
Regarding the solution of equation \eqref{3.2} driven by $p$-rough path $x$, we will always first consider solutions of the equations
\begin{equation}\label{approximation of rough differential eqn}
y_{n}(t)=y(0)+\sum^{d}_{j=0}\int^{t}_{0}V_{j}(y_{n}(s))dx^{j}_{n}(s), \quad 0\le t \le T.
\end{equation}
Then $y_{n}(t)$ converges in p-variation to some $y \in C^{p-var}([0,T], \mathbb{R}^{d})$ and y is called the solution of the rough differential equation \eqref{3.2}. $($See \cite{Ly}[Theorem 3.7, Theorem 3.10] for details$)$. 
\end{remark}
Following the Taylor expansion idea in \cite{FBb}[Definition 2.2], we can define the Taylor expansion of $y(t)$ by iterative application of the change of variable formulas. We first apply this strategy to \eqref{approximation of rough differential eqn} with
$x_n\in C^{1-var}([0,T],\mathbb{R}^d)$.
Then, it is clear that we could also apply the iterative process for a $p$-rough path $x$ by the approximation argument mentioned above. 
\begin{definition}\label{taylor expansion}
The Taylor expansion associated with the differential equation $($\ref{3.2}$)$ is defined as
\begin{align*}
y_0+\sum^{\infty}_{k=1}g_{k}(t),
\end{align*}
where 
\begin{align*} 
g^{j}_{k}(t)=\sum_{|I|=k}P^{j}_{I}\int_{\triangle^{k}[0,t]}dx^{I}, \quad P^{j}_{I}=(V_{i_1}\cdots V_{i_k}\pi^{j})(y_0).
\end{align*}
We denote $\pi_j(y)=y^j$ as the $j$-th projection map. And we keep the convention that $V_{i_1}\cdots V_{i_k}= (V_{i_1}\cdots(V_{i_{k-2}}(V_{i_{k-1}} V_{i_k}))\cdots)$, which is non-associative and represents the iterative operation of the vector fields in order.
\end{definition}
\subsection{Convergence of the Taylor expansion}
\subsubsection{A general convergence result}
In the following, we follow the proof of\\ \cite{FB3}[Theorem 2.4] to get our convergence result.
\begin{proposition}\label{corollary1}
Let $y^{\varepsilon}(t)$ be the solution to the scaled differential equations of \eqref{3.2} as below:
\begin{align}\label{rescalled_de}
\begin{cases}
   dy^{\varepsilon}(t)=\displaystyle \sum^{d}_{i=0}\varepsilon V_{i}(y^{\varepsilon}(t))dx^{i}_{t}\\
   y^{\varepsilon}_0=y_0.
\end{cases}
\end{align}
Then for every $t\in [0,T]$, and for every fixed $\varepsilon$, the map $\varepsilon \rightarrow y^{\varepsilon}(t)$ is $C^{\infty}$. If $y_{n}^{\varepsilon}(t)$ denotes the solution to the scaled differential equations driven by $x_{n}(t)$, then $\displaystyle \frac{\partial y_{n}^{\varepsilon}(t)}{\partial \varepsilon}$ converges to $\displaystyle \frac{\partial y^{\varepsilon}(t)}{\partial \varepsilon}$ uniformly on $[0,T]$.
\end{proposition}
\begin{proof}
	According to \cite{FB3}[Prop. $2.3$], the proof relies on the smoothness of the solution of equation driven by $p$-rough path with respect to its initial condition, which follows from \cite{Pfriz}[Prop. 11.3]. 
\end{proof}
\begin{theorem}\label{general}
Let $y_0+\sum^{\infty}_{k=1}g_{k}(t)$ be the Taylor expansion associated with equation $($\ref{rough differential equation}$)$ which is defined in Definition \ref{taylor expansion}. There exists $T >0$, such that for every  $t\in (0, T)$, the series
\begin{align*}
 &\sum^{\infty}_{k=1}\|g_{k}(t)\|
 \end{align*}
 is convergent and 
 \begin{align*}
y(t)=y_0+ \sum^{\infty}_{k=1}g_{k}(t).
\end{align*}
\end{theorem}
\begin{proof}
Let us fix $\rho >0$. Consider the parameterized differential equation \eqref{rescalled_de}, due to the analyticity of the vector fields $V_i$'s, there exists strictly positive time
\[
T_C(\rho)=\inf_{\varepsilon, |\varepsilon| <\rho} \{ t \ge 0:  y^{\varepsilon}(t) \notin B(y_0,C/2) \}.
\]

By proposition \ref{corollary1}, for a fixed time $t \ge 0$, the map $\varepsilon \to y^{\varepsilon}(t)$ is $C^{\infty}$. Denote $y^{\varepsilon}_{n}(t)$ as the approximating sequence to the solution $y^{\varepsilon}(t)$, we know that $y^{\varepsilon}_{n}(t)\rightarrow y^{\varepsilon}(t)$ uniformly on $[0,T]$. It then follows that there exists $N>0$, such that when $n>N$ and for every $t\in [0,1]$, we have $\|y^{\varepsilon}_{n}(t)-y^{\varepsilon}(t)\|<C/2$. On the other hand, when $t<T_C(\rho)$, we have $\|y^{\varepsilon}(t)-y_{0}\|<C/2$. Therefore, when $n>N$ and $t<T_C(\rho)$, we get $y^{\varepsilon}_{n}(t)\in B(y_{0},C)$.\\
Now for the case where $n>N$, we could consider the following complex differential equations:
\[
\begin{cases}
   dy^{z}_{n}(t)=\displaystyle \sum^{d}_{i=0}z V_{i}(y^{z}_{n}(t))dx^{i}_{n}(t),\\
   y^{z}_n(0)=y_0,
\end{cases}
\]
where we take the holomorphic extension of the vector fields $V_i$'s to a ball around $y_0 \in  \mathbb R^n \subset \mathbb C^n$ and the above equation is well defined up to time $T_C(\rho)>0$. We claim that the map $z \to y^{z}_{n}(t)$ is not only $C^{\infty}$ smooth but also analytic. Differentiating with respect to $\bar{z}$, the integral expression immediately gives $\frac{\partial y^{z}_{n}(t)}{\partial \bar{z}}=0$ by the uniqueness of solutions of linear equations. That is, $z \to y_n^{z}(t)$ is analytic on the disc  $ |z| <\rho $. It follows by Proposition \ref{corollary1} that $\frac{\partial y^{z}(t)}{\partial \bar{z}}=0$. Therefore, the map $z\rightarrow y^{z}(t)$ is analytic and $y^{z}(t)$ admits a Taylor series of $z$ when $t<T_C(\rho)$. In particular, observe that we could consider the Taylor expansion
 of $f(z)\triangleq y^z(t) $ at $z=1$. Hence,
 \[
 f(1)=\sum_{n=1}^{\infty} \frac{f^{(n)}(0)}{n!},
 \]
 which directly gives us $ g_n(t)=\frac{f^{(n)}(0)}{n!}$ since $f(z)=x_0+\sum_{n=1}^{\infty}z^ng_n(t)$. Thus we can choose $\rho > 1$, which finishes the proof. 
\end{proof}
\subsubsection{Quantitative bounds}
\begin{lemma}\label{bound}
Let $\gamma$ be a constant such that $0<\gamma<\beta $. There exists a constant $K_{\beta,\gamma} >0$ such that for every $N \ge 0$ and $ x \ge 0$,
\[
\sum_{k=N+1}^{+\infty} \frac{\Gamma(k \gamma)}{\Gamma(k\beta)} x^{k-1} \le K_{\beta,\gamma}
\begin{cases}
e^{2 x^{\frac{1}{\beta-\gamma}}} , \quad \text{if } N = 0, \\
 \frac{x^N e^{2 x^{\frac{1}{\beta-\gamma}}}}{\Gamma((\beta-\gamma)N)}, \quad \text{if } N \ge 1,
\end{cases}
\]
where
\[
\Gamma(t)=\frac{1}{t}\Pi_{n=1}^{\infty}\frac{(1+\frac{1}{n})^t}{1+\frac{t}{n}}.
\]
\end{lemma}
\begin{proof}
We make the proof for $N \ge 1 $ and let the reader adapt the argument when $N=0$. We have
\begin{align*}
\sum_{k=N+1}^{+\infty} \frac{\Gamma(k \gamma)}{\Gamma(k\beta)} x^{k-1} &=x^N \sum_{k=0}^{+\infty} \frac{\Gamma((k+N+1) \gamma)}{\Gamma((k+N+1)\beta)} x^{k} \\
 & =\frac{x^N}{\Gamma((\beta-\gamma) N)} \sum_{k=0}^{+\infty} \frac{\Gamma((k+N+1) \gamma)\Gamma((\beta-\gamma)N)}{\Gamma((k+N+1)\beta)} x^{k}  \\
 &\le K_{\beta,\gamma} \frac{x^N}{\Gamma((\beta-\gamma)N)} \sum_{k=0}^{+\infty} \frac{x^k}{\Gamma( (k+1)(\beta-\gamma)) }.
\end{align*}
We conclude the proof by using the fact 
from \cite{Lyons} and for every $x \ge 0$, we have
\[
\sum^{\infty}_{k=0}\frac{x^{k}}{\Gamma((1+k)(\beta-\gamma))}\leq \frac{4e^{2}}{\beta-\gamma}e^{2x^{\frac{1}{\beta-\gamma}}}.
\] 
\end{proof}
Below is our result on the convergence rate of the Taylor expansion.
\begin{theorem}\label{convergence}
Let $p\geq 1$ and assume that there existse $M>0$ and $0<\gamma<\frac{1}{p}$ such that for every word $I \in \{ 0,\cdots, d\}^k$, we have
\begin{align}\label{criteria}
\|P_{I}\|\leq \Gamma(\gamma|I|)M^{|I|}.
\end{align}
For $r>1$, we define $\displaystyle T_{C}(r)=\inf\{t: \sum^{\infty}_{k=1}r^{k}\|g_{k}(t)\|\geq C/2\}$. Then, we conclude,
\begin{enumerate}
\item There exists $T_{C}(r)>0$ and when $t<T_{C}(r)$, we have
$y(t)=y_0+\sum^{\infty}_{k=1}g_{k}(t)$.
\item For any $N>1$, there exists a constant $Q_{p,\gamma,M,T}>0$ depending on the subscript variables such that when $t<T_{C}(r)$,
\[
\left\|y(t)-\left(y_0+\sum^{N}_{k=1}g_{k}(t)\right) \right\| \le Q_{p,\gamma,M,T}\frac{\left(MK\omega(0,T)^{1/p}\right)^{N}}{\Gamma((\frac{1}{p}-\gamma)N)}e^{2(MC\omega(0,T)^{\frac{1}{p}})^{\frac{p}{1-p\gamma}}}.
\]
where 
\[
 \omega(0,t)=\left(\sum^{[p]}_{j=1}\left\| \int dx^{\otimes j}\right\|^{1/j}_{\frac{p}{j}-var, [0,t]}\right)^{p}.
\]
\end{enumerate}
\end{theorem}

\begin{proof}
Let $y(t)$ be the solution of RDE (\ref{3.2}) with the associated Taylor series $y_0+\sum^{\infty}_{k=1}g_{k}(t)$, we know that there exists a sequence $y_{n}(t)$ that converges to $y(t)$ in $p$-variation topology. We denote by $y_{0}+\sum^{\infty}_{k=1}g_{k,n}(t)$ the Taylor series associated with RDE satisfied by $y_{n}(t)$. We first show that $T_{C}(r)> 0$. According to \cite{Ly}, we have
\begin{align*}
\sum_{k=1}^{\infty}r^{k}\|g_{k}(t)\|\leq &\sum_{k=1}^{\infty}r^{k}\sum_{I=(i_1,\cdots,i_{k})}\|P_{I}\|\left|\int_{\triangle^{k}_{[0,t]}}dx^{I}\right|\\
 \leq &\sum_{k=1}^{\infty}r^{k}\Gamma(\gamma k)M^{k}\|\int_{\triangle^{k}[0,t]}dx^{\otimes k}\|\\
   \leq& \sum_{k=1}^{\infty}B(\gamma,1/p-\gamma)\frac{(rMK\omega(0,T)^{\frac{1}{p}})^{k}}{\Gamma(k(\frac{1}{p}-\gamma))}.
\end{align*}
and $T_{C}(r)>0$ follows by the fact that $0<\gamma<1/p$. The same type of estimates also hold for $g_{k,n}(t)$ for any $n\in \mathbb{N}_+$, and it follows by the Dominated Convergence Theorem that  $y_0+\sum^{\infty}_{k=1}g_{k,n}(t)\rightarrow y_0+\sum^{\infty}_{k=1}g_{k}(t)$ uniformly on $[0,T]$. On the other hand, when $t<T_{C}(r)$, we have: $\sum^{\infty}_{k=1}\varepsilon^{k}\|g_{k}(t)\|<C/2$ for any $0<\varepsilon \le r$. Therefore, there exists $N>0$ such that when $n>N$, we have
$\sum^{\infty}_{k=1}\varepsilon^{k}\|g_{k,n}(t)\|<C.$
Following Definition \ref{taylor expansion}, for any $n>N$, we have  $\displaystyle y_{n}(t)=y_0+\sum^{\infty}_{k=1}g_{k,n}(t)$. Then the equality
$y(t)=y_0+\sum^{\infty}_{k=1}g_{k}(t)$ 
follows from the fact that $y_{n}(t)\rightarrow y(t)$ and  $\displaystyle y_0+\sum^{\infty}_{k=1}g_{k,n}(t)\rightarrow y_0+\sum^{\infty}_{k=1}g_{k}(t)$ uniformly on $[0,T_{C}(r)]$. 
Now let us consider the error estimate. When $t<T_{C}(r)$, we have
\begin{align*}
\left\|y(t)-\left(y_0+\sum^{N}_{k=1}g_{k}(t)\right) \right\|
&=\left\|\sum^{\infty}_{k=N+1}g_{k}(t)\right\|\leq \sum^{\infty}_{k=N+1}\sum_{|I|=k}\|P_{I}\|\left|\int_{\triangle^{k}[0,t]}dx^{I}\right|\\
&\leq \sum^{\infty}_{k=N+1} \Gamma(\gamma k)M^{k}\|\int_{\triangle^{k}[0,t]}dx^{\otimes k}\| \\
&\leq M\omega(0,T)^{\frac{1}{p}}K\sum^{\infty}_{k=N+1}\frac{\Gamma(k\gamma)}{\Gamma(\frac{k}{p})}\left(MK\omega(0,T)^{\frac{1}{p}}\right)^{k-1}\\
&\leq  Q_{p,\gamma,M,T}\frac{\left(MK\omega(0,T)^{\frac{1}{p}}\right)^{N}}{\Gamma((\frac{1}{p}-\gamma)N)}e^{2(MK\omega(0,T)^{1/p})^{\frac{p}{1-p\gamma}}}.
\end{align*}
The last inequality follows from Lemma \ref{bound}.
\end{proof}
\begin{example}
A non-trivial example where Theorem \ref{convergence} applies can be found in  \cite[section 2.4]{FB3}, where the following equation is considered on a connected Lie group $\mathbb G$ with its Lie algebra $\mathfrak{g}$.
\[
\begin{cases}
&dy(t)=\sum_{i=0}^dV_i(y(t))dX^i(t),\\
&y_0=e.
\end{cases}
\]
Where $e$ is the identity element of $\mathbb G$, and $V_0,V_1,\cdots, V_d\in \mathfrak g$ are analytic left invariant vector fields. The only difference is that we consider  $x$ as $p$-rough path with $p>2$ satisfying Hypothesis \eqref{hypo: vector-path} $(ii)$.
\end{example}
\section{Castell expansion and tail estimate for RDE.}
We first introduce the following concept for general Gaussian process and fractional Brownian motion (fBm).  For more details, we refer to \cite{FBGTO}[Section 2]. On a complete probability space $(\Omega,\mathcal{F},\mathbb{P})$, we denote $X(t)=(X^1(t),\cdots,X^d(t))$
as a continuous, centered Gaussian process with i.i.d. components. Its covariance function $R$ has the form of
\begin{align*}
R_{X}\left(\begin{array}{c}
s,t\\
u,v
\end{array}\right)\equiv
R^{st}_{uv}\equiv \mathbb{E}[(X^1_t-X^1_s)(X^1_v-X^1_u)].
\end{align*}
 The $2D$ $\rho$-variation of $R$ on rectangle $[s,t]^2$ is defined as
\begin{equation*}
V_{\rho}(R;[s,t]^2)\equiv\sup \lbrace (\sum_{i,j}|R^{t_i,t_{i+1}}_{s_j,s_{j+1}}|^{\rho})^{1/\rho}; (s_j),(t_i)\in \mathcal{D}([s,t])\rbrace.
\end{equation*}
 We denote $V_{\rho}(R)=V_{\rho}(R;[0,1]^2)$ for simplicity. In particular, $V_{\rho}(R)$ is a special case, recovered as $V_{\rho}=V_{\rho,\rho}
$ for the mixed right $(\gamma,\rho)$-variation
given in (\cite{PGGS}) defined as below: for $\gamma,\rho\geq1$,  
\begin{align*} 
V_{\gamma,\rho}(R_{X};[s,t]\times [u,v]):=\sup_{\substack{(t_{i})\in\mathcal{D}([s,t])\\
(t_{j}^{\prime})\in\mathcal{D}\left(\left[u,v\right]\right)
}
}\left(\sum_{t'_{j}}\left(\sum_{t_{i}}\left|R_{X}\left(\begin{array}{c}
t_{i},t_{i+1}\\
t_{j}^{\prime},t_{j+1}^{\prime}
\end{array}\right)\right|^{\gamma}\right)^{\frac{\rho}{\gamma}}\right)^{\frac{1}{\rho}}.\label{eq:mixed_var}
\end{align*}
 The Cameron-Martin space $\bar{\mathcal{H}}$ associated with the Gaussian process $X(t)$ is defined to be the completion of the linear space of functions of the form
\[
\sum_{i=1}^{n}a_{i}R\left(  t_{i},\cdot\right)  ,\quad a_{i}\in%
\mathbb{R}
\text{ and }t_{i}\in\left[  0,T\right]  ,
\]
with respect to the inner product induced by 
$
\left\langle R\left(  t_{i},\cdot\right)  ,R\left(  s_{j},\cdot\right)  \right\rangle _{\mathcal{H}}=R\left(  t_{i},s_{j}\right)  .
$
The embedding coefficient from the Cameron-Martin space $\bar{\mathcal{H}}$ to the space of continuous functions with finite $q$-variation is defined as $C_{emb}(T)$ with $1/p+1/q>1$, such that $\forall$ $h\in \bar{\mathcal{H}}$,
\[
|h|_{q-var;[0,T]}\leq C_{emb}(T)|h|_{\mathcal{\bar{H}}}.
\]
In particular, $C_{emb}(T)=\sqrt{V_{1,\rho}(R_{X};J\times J)}$ with $J=[0,T]$. See details about $C_{emb}(T)$ in \cite{PGGS}.

As a special Gaussian process, we define $B(t)=(B^1(t),\cdots,B^d(t))$ a d-dimensi-\\onal fractional Brownian motion (fBm) indexed by $[0,1]$ with Hurst parameter $H>1/4$, if the components $B^i$ are $i.i.d.$ and each component $B^i$ is a centered Gaussian process satisfying
\begin{equation*}
\mathbb{E}[(B^i(t)B^i(s))]=\frac{1}{2}(|t|^{2H}+|s|^{2H}-|t-s|^{2H}),
\end{equation*}
for $s,t\in[0,1]$. The embedding coefficient defined above for fBm has the form of $(C_{emb}(t))^2=t^{2H}.$ 

   \begin{remark}
The results from \cite{Pfriz}[Theorem 15.42] tells us that for $X(t)$ as a continuous, centered Gaussian process, if $X(t)$ has finite $2D$-$\rho$-variation for $\rho\in[0,2)$, then $X(t)$ has a lift to a geometric $p$-rough path provided $p>2\rho$. Moreover, there is a unique natural lift which is the limit, in the $p$-var topology, of the canonical lift of piecewise linear approximations to $X$. The same type of results can also be found for fBm in \cite{Qian}. Thus according to Remark \ref{remark:approximation}, by using linear approximation to general Gaussian process and fBm, the results from the previous section applies to RDE \eqref{3.2} by changing the driving path to be general Gaussian process and fBm.
   \end{remark}
\subsection{Asymptotic expansion and tail estimate for Castell expansion.}
We present Theorem \ref{thm:gaussian asymptotic and tail estimate} for the scaled differential equations \eqref{scaled gaussian sde} below driven by general Gaussian process. The argument for fBm with $H>1/4$ follows similarly in the remark below. To be consistent with the Taylor expansion in the previous section, we consider a more general scaling RDE as follows, 
\begin{equation}
\label{scaled gaussian sde}
dy^{\varepsilon}(t)=V_0(\varepsilon, y^{\varepsilon}(t))dt+ \sum_{i=1}^d\varepsilon V_i(y^{\varepsilon}(t))dx^i(t),
\end{equation} 
where we denote $V_0(\varepsilon, y^{\varepsilon}(t))$ as a general drift term with any potential scaling power, namely, we have $ V_0(\varepsilon, y^{\varepsilon}(t))=\varepsilon^{\mathcal S}V_0(y^{\varepsilon}(t))$. If the driving process $x$ is fraction Brownian motion with Hurst parameter $H$, then the scaling power $\mathcal S=H^{-1}.$

 We then introduce the following general order definition.
\begin{definition}
	For every word $I\in\{0,1,\cdots,d\}^k$, for $k\in\mathbb Z_+$, we denote $\sharp(I)$ as the number of zeros in word $I$. Then, for any scaling power $\mathcal S$, we denote $\mathcal O(I)$ as the order of $I$, defined as $\mathcal O(I):=\mathcal S*\sharp(I)+(|I|-\sharp(I))$, where $|I|=k$ is the size of $I$. In particular, if $\mathcal O(I)$ is not an integer, we will take its integer part, i.e. $\mathcal O(I):=[\mathcal O(I)]$.
\end{definition}
\begin{theorem}\label{thm:gaussian asymptotic and tail estimate}
Assume that Hypothesis \ref{hypo: vector-path}$(i)$ is in force. If $x(t)=(x^1(t),\cdots,\\x^d(t)):[0,T]\rightarrow \mathbb{R}^d$ is continuous, centered Gaussian process with i.i.d. components and finite 2-dimensional $\rho$-variation, for $\rho\in[0,2)$, and $x^0(t)=t$. Let $y^{\varepsilon}(t)$ denote the solution of the scaled differential equation \eqref{scaled gaussian sde} with initial value $y_0$. 
Then there exists a random time $T_C(r)>0$, such that for every $\varepsilon<r$, every positive integer $N\in \mathbb Z_+$ and every $t<T_C(r)$, we have
\begin{equation}\label{castel general form}
y^{\varepsilon}(t)=\exp\left(\sum_{I:\mathcal O(I)\le N} \varepsilon^{\mathcal O(I)}\Lambda_I(x)_tV_I\right)(y_0)+\varepsilon^{N+1}R_{N+1}(\varepsilon,t),~a.s.,
\end{equation}
where we denote $\exp(W)(y_0)$ as the time one map of the 
flow generated by a general vector field $W$ starting at $y_0$. 
Furthermore, there exist constants $\alpha,c>0$, such that for every $\tau\in(0,T_C(r))$ and every $\xi\geq 1,$
\begin{equation}
\label{condition:gaussian tail estimate }
\mathbb{P}\Big(\sup_{t\in [0,\tau]}\|R_{N+1}(\varepsilon,t)\|\geq \xi;\quad \tau < T_C(r)\Big)\leq \exp\left(-c\frac{\xi^{\alpha}}{(C_{emb}(\tau))^2}\right).
\end{equation}
\end{theorem}
\begin{remark} If we take the same scaling SDE as we did in Section 2, where $\mathcal S=1$, then
	\begin{equation*}
y^{\varepsilon}(t)=\exp\left(\sum_{k=1}^{N}\varepsilon^{k}\sum_{I: |I|=k}\Lambda_I(x)_tV_I\right)(y_0)+\varepsilon^{N+1}R_{N+1}(\varepsilon,t),~a.s..
\end{equation*}

\end{remark}
\begin{remark}\label{fBm remark}
When the driving path $ x:[0,T]\rightarrow \mathbb{R}^d$ is a fBm with $H>1/4$. The above Theorem \ref{thm:gaussian asymptotic and tail estimate} can be simplified with $(C_{emb}(t))^2=t^{2H}$ and $\alpha=(2H+1)\wedge 2.$ Let $\varepsilon=1$, we have $x^{\varepsilon}(t)=\varepsilon^{-1}x(\varepsilon^{H^{-1}}t)$, then following \cite{Aze}[Section 5.2] $($refer \cite{Tindel} for the scaling property$)$ we have,
\begin{align}
\label{scaling case}
&y(t)=\exp\left(\sum_{I:\mathcal O(I)\le N }\Lambda_I(x)V_I\right)(y_0)+t^{H(N+1)}R_{N+1}(t),      
\end{align}
and 
\begin{align}\label{scaling case tail estimate}
\mathbb{P}\Big(\sup_{t\in [0,\tau]}\|t^{(N+1)H}R_{N+1}(t)\|\geq \xi\tau^{(N+1)H};\quad \tau<T_C(r)\Big)\leq \exp\Big(-\frac{c_H \xi^{(2H+1)\wedge 2}}{\tau^{2H}}\Big).
\end{align}
\end{remark}
\begin{remark}
The same type of differential equations \eqref{scaled gaussian sde} driven by fBm with $H>1/4$ is also considered in \cite[Theorem 9]{FB1}, they proved the representation of  $\mathbb E(f(y(t)^{y_0}))=P_t f(y_0)$ at small time. The current results and those in \cite{FB1} are closely related to cubature methods \cite{LV02}.
\end{remark}
\subsection{Proof of the main result}
We first present the following lemma to prepare us ready for the proof of the main results. 
\begin{lemma} \label{lemma expansion}
Denote $y^{\varepsilon}(t)$ as the solution of the scaled RDE \eqref{scaled gaussian sde} driven by path $x$ with bounded $p$-variation. Let $N\ge 1$. Consider the function $F:\mathbb R\rightarrow \mathbb R^n$ defined by 
 $$F(\varepsilon)\triangleq \exp\Big(\sum_{I: \mathcal O(I)\le N}\varepsilon^{\mathcal O(I)} \Lambda_I(x)_tV_I  \Big)(y_0),\quad \varepsilon\in \mathbb R, $$
 where $\Lambda_I(x)_t$ is defined in \eqref{lambda notation}. Then for each $1\le k\le N$,
\[
\frac{F^{(k)(0)}}{k!}=\sum_{I:\mathcal O(I)=k}P_I\cdot \int_{\triangle^{k}[0,t]}dx^{I},
\]
where $P_I=(V_{i_1}\cdots V_{i_{k}}I)(y_0)$ is defined in Definition \ref{taylor expansion}.
\end{lemma}
\begin{proof}
Similar to \cite{FB3}[Theorem 2.11], combing with the new scaling order $\mathcal O(I)$, the solution of equation \eqref{scaled gaussian sde} can be represented as 
$$y^{\varepsilon}(t)=\exp\Big(\sum_{I:\mathcal O(I) =1}^{\infty}\varepsilon^{\mathcal O(I)}\Lambda_I(x)_tV_I  \Big)(y_0),\quad \varepsilon\in \mathbb R, $$
which equals $\sum_{k=0}^{\infty}\varepsilon^kg_k(t)$ according to the Taylor expansion in Definition \ref{taylor expansion}. For each $N\ge 1$, take truncation of the order of $\varepsilon$ up to $N$, by differentiating both quantities and taking $\varepsilon=0$, we thus complete the proof.  
\end{proof}

Next, we recall the following definition introduced in \cite{Aze}.\\ 
$\mathbf{Definition~ of~} \omega(\alpha,c,\zeta)$.\\
Let $\zeta$ be a random time, $y(t)$ is said to be in the family of $\omega(\alpha,c,\zeta)$ if and only if: for every $R\ge c$,
\begin{equation}
\label{omega}
\mathbb{P}\Big(\sup_{0\leq s\leq t}\|y(s)\|\geq R;~t<\zeta\Big)\leq \exp\Big(-\frac{R^{\alpha}}{ct}\Big).
\end{equation}
With the definition of $\omega(\alpha,c,\zeta)$ defined above, the following properties (see \cite{Aze}) hold true:\\
\label{P1}
$\mathbf{(P1)}$.  Let $\phi(t)$ be a continuous process on $[0,\zeta]$ with values in the space of the polynomials of degree less than $q$, in $p$ Euclidean variables, with coefficients bounded by some constant $A$ on $[0,\zeta]$. The image of $\mathcal{W}(\alpha_1,c_1,\zeta)\times\cdots\times\mathcal{W}(\alpha_p,c_p,\zeta)$ by the mapping 
\begin{equation}
(X^1,\cdots,X^p)\rightarrow Y,\qquad y(t)=\phi_t(X^1(t),\cdots, X^p(t)),
\end{equation}
is in some $\mathcal{W}(\alpha,c,\zeta)$ ,with $\alpha,c$ determined by $A,p,q,\alpha_1,c_1,\cdots,\alpha_p,c_p.$\\
\label{P2}
$\mathbf{(P2)}.$ If $\zeta$ is bounded by some fixed $T$, the image of $\mathcal{W}(\alpha,c,\zeta)$ by the mapping $X\rightarrow Y$ with $Y=\int_0^tX(u)dB(u)$, is in some $\mathcal{W}(\alpha,c,\zeta)$. This is also true for the mapping $X\rightarrow Y$, with $ y(t)=\int_0^tX(u)du$. The driving signal $B(u)$ can be Brownian motion, fBm with $ H>1/4$ and general Gaussian process with finite $2D~\rho-$variation according to Proposition \ref{2.10 FB} and Proposition \ref{gaussian tail} in the Appendices.

Given the above definition, according to Theorem \ref{general}, we know that equation (\ref{scaled gaussian sde}) admits a unique solution in the rough path sense denoted as $y^{\varepsilon}(t)$, and the solution $y^{\varepsilon}(t)$ can be represented as below
\begin{equation}
\label{$g_k(t)$}
y^{\varepsilon}(t)=y_0+\sum_{k=1}^{\infty
}\varepsilon^kg_k(t)=y_0+\sum_{k=1}^{N}\varepsilon^{k}g_k(t)+\varepsilon^{N+1} M_{N+1}(\varepsilon,t).
\end{equation}

We then introduce the map $\phi$ (see \cite{Cast}) defined for an appropriate $d$ by 
\[
\phi:\mathbb{R}^d\rightarrow \mathbb{R}^n\quad \text{with}\quad 
(\Lambda_I)_{\mathcal O(I)<N+1} \mapsto \exp \Big(\sum_{I: \mathcal O(I)<N+1}\Lambda_IV_I\Big)(y_0).
\]

 It is clear that by Taylor expansion, we have the following:
\begin{align}
\label{$h_k(t)$}
\exp \Big(\sum_{\mathcal O(I)=1}^{N}\varepsilon^{\mathcal O(I)}\sum_{I}\Lambda_I(t)V_I\Big)(y_0)&=\phi((\varepsilon^{\mathcal O(I)}\Lambda_I)_{\mathcal O(I)<N+1})\nonumber \\
&=y_0+\sum_{k=1}^{N}\varepsilon^k h_k(t)+\varepsilon^{N+1} P_{N+1}(\varepsilon,t).
\end{align}
According to Lemma \ref{lemma expansion}, we directly get the following lemma.	
\begin{lemma} For any $N\ge 1$, we have 
\label{lemma g-k=h-k}
 $g_k(t)=h_k(t)$, for every $k=1,2,\cdots,N$.
\end{lemma}
Furthermore, we have the following tail estimates. 
\begin{lemma}
\label{lemma M and g}
There exists some random time $\zeta>0$, such that 
$g_j(t)\in\omega(\alpha_j,c_j,\zeta)$, $j=1,\cdots,N$, and $M_{N+1}(\varepsilon,t)\in \omega(\alpha_M,c_M,\zeta)$, namely
\begin{equation}
\label{g_j}
\mathbb{P}\Big(\sup_{t\in [0,\tau]}\|g_j(t)\|\geq \xi;~\tau\leq \zeta\Big)\leq \exp \Big(-c_j \frac{\xi^{\alpha_j}}{(C_{emb}(\tau))^2}\Big),
\end{equation}
and 
\begin{equation}
\label{M_N gaussian}
\mathbb{P}\Big(\sup_{t\in [0,\tau]}\|M_{N+1}(\varepsilon,t)\|\geq \xi;~ \tau<\zeta\Big)\leq \exp\Big(-c_M\frac{\xi^{\alpha_M}}{(C_{emb}(\tau))^2}\Big),
\end{equation}
for some constants $\alpha_M,\alpha_j,c_M,c_j$, for $j=1,\cdots,N$. 
\end{lemma}
\begin{lemma}
\label{lemma P} There exists some random time $\zeta>0$, such that
$P_{N+1}(\varepsilon,t)\in \omega(\alpha_P,c_P,\zeta)$, i.e.
\begin{equation}
\label{P_N gaussian}
\mathbb{P}\Big(\sup_{t\in [0,\tau]}\|P_{N+1}(\varepsilon,t)\|\geq \xi;~ \tau<\zeta\Big)\leq \exp\Big(-c_P\frac{\xi^{\alpha_P}}{(C_{emb}(\tau))^2}\Big),
\end{equation}
for some constants $\alpha_P$ and $c_P$.
\end{lemma}
With the aforementioned lemmas, we are now ready to prove our main theorem.\\
$\mathbf{Proof~of~ theorem ~\ref{thm:gaussian asymptotic and tail estimate} }$.

\begin{proof}
Based on Lemma \ref{lemma g-k=h-k}, subtracts (\ref{$g_k(t)$}) by (\ref{$h_k(t)$}) we have the following:
\begin{equation}
\label{result}
y^{\varepsilon}(t)=\exp \Big(\sum_{\mathcal O(I)=1}^{N}\varepsilon^{\mathcal O(I)}\sum_{I}\Lambda_I(t)V_I\Big)(y_0)+\varepsilon^{N+1} R_{N+1}(\varepsilon,t).
\end{equation}
Then  according to Lemma \ref{lemma M and g} and Lemma \ref{lemma P}, we deduce almost surely that  $R_{N+1}(\varepsilon,t)=M_{N
+1}(\varepsilon,t)-P_{N+1}(\varepsilon,t)$, which means $R_{N+1}(\varepsilon,t) \in \omega(\alpha_R,c_R,\zeta)$, for some random time $\zeta$ and some constants $\alpha_R$, $c_R$. More precisely, we have 
\[
\mathbb{P}\Big(\sup_{t\in [0,\tau]}\|R_{N+1}(\varepsilon,t)\|\geq \xi;~\tau <\zeta\Big)\leq \exp\Big(-c_R\frac{\xi^{\alpha_R}}{(C_{emb}(\tau))^2}\Big),
\]
where $\alpha_R,c_R$ depend on $\alpha_M,\alpha_P,c_M,c_P.$ The proof is thus completed.
\end{proof}

\begin{remark}
The above proof applies to the fBm case by using Proposition \ref{2.10 FB} $($see appendices$)$ instead of Proposition \ref{gaussian tail} with $(C_{emb}(t))^2=t^{2H}$ for $H>1/4$. 
\end{remark}

\par\bigskip\noindent
{\bf Acknowledgment.} We would like to thank our advisor Professor Fabrice Baudoin for suggesting the
problem and for all his help during our work.

\begin{appendices}
Consider the equation below,
\begin{equation}
\label{X(t)}
Y(t)=y_0+\int_0^tV_0(Y(s))ds+\sum_{i=1}^d\int_0^tV_i(Y(s))dx_s^i.
\end{equation}
We have the following propositions.
\begin{proposition}\label{app: fBm}
\label{2.10 FB}\cite{FB2}[Proposition 2.10] For some constant $c_H$, we have 
\begin{equation}
\label{y(t)}
\mathbb{P}\Big(\sup_{t\in[0,T]}|Y(t)-y_0|\geq \xi\Big)\leq \exp\Big(-\frac{c_H\xi^{(2H+1)\wedge 2}}{t^{2H}}\Big),
\end{equation}
where $Y(t)$ is the solution of \eqref{X(t)} together with the driving path and the vector fields satisfy the assumptions in Remark \ref{fBm remark}.
Estimate (\ref{y(t)}) generalize the Brownian motion case in \cite[Appendix.2]{Aze} to fBm with $H>1/4$. 
\end{proposition}
\begin{proposition}\label{app: gaus}
\label{gaussian tail}
The following inequality holds true,
\begin{align} 
\label{tilde X}
\mathbb{P}\Big(\sup_{t\in[0,T]}|\tilde Y(t)-y_0|\geq \xi\Big)\leq \exp\Big(-c\frac{\xi^{2/q}}{(C_{emb}(t))^2}\Big),
\end{align}
where $\tilde Y(t)$ is the solution of equation (\ref{X(t)}) together with the driving path and the vector fields satisfy the assumptions in Theorem \ref{thm:gaussian asymptotic and tail estimate}.
\end{proposition}
\begin{remark}
The proof was given in the previous version of our paper and we now omit the proof for conciseness and refer to \cite{FBGTO}[Proposition 3.7] for details. To make the connections, our $\frac{2}{q}$ is the same as $1+\frac{1}{\rho}$ and our $C_{emb}(t)$ is the same as $\kappa_t$ in \cite{FBGTO}.
\end{remark}


\begin{proof}[\textbf{Proof of lemma \ref{lemma M and g}}]
We follow the strategy presented in Azencott \cite{Aze}, where the driving process is Brownian motion. We
consider the general scaling rough differential equation as follows,
\begin{align}
\label{Aze}
dy^{\varepsilon}(t)&=V_0(\varepsilon,y^{\varepsilon}(t))dt+ \sum_{i=1}^d\varepsilon V_i(y^{\varepsilon}(t))dx^i(t),\nonumber\\
&=b(\varepsilon,y^{\varepsilon}(t))dt+\varepsilon \cV(y^{\varepsilon}(t))dx(t),
\end{align}
where we denote $\cV=[V_1,\cdots,V_d]\in\mathbb R^{n\times d}$, $V_0(\varepsilon, y^{\varepsilon}(t) )=b(\varepsilon, y^{\varepsilon}(t) )$ and $x(t)=[x^1(t),\cdots,x^d(t)]\in \mathbb R^{d\times 1}$, and we denote $y_0\in\mathbb R^n$ as the initial value of $y^{\varepsilon}(t).$ 
 With a little abuse of notation, we define
\begin{align}
\label{sigma ij}
\cV_{ij}(t)=\frac{1}{i!j!}\frac{\partial^{i+j}\cV}{\partial \varepsilon^i \partial x^j}(0,t,y_0),\quad 
b_{ij}(t)=\frac{1}{i!j!}\frac{\partial^{i+j}b}{\partial \varepsilon^i \partial x^j}(0,t,y_0).
\end{align}
In particular, for each $j\geq 1$, $y\in \mathbb{R}^n$, $\cV_{ij}(t)y^j$ and $b_{ij}(t)y^j$ are valued on the point $(y,\cdots,y)\in (\mathbb{R}^n)^{\otimes j}$. The Taylor series of $\cV(t,\cdot)$ and $b(t,\cdot)$ at $0$ and $y_0$ are  written as 
\begin{align}
\label{sigma b}
\cV(t,\varepsilon,y)&=\sum_{0\leq i,j}\varepsilon^i\cV_{ij}(t)\cdot(y-y_0)^j=\sum_{i+j\leq N}\varepsilon^i\cV_{ij}(t)\cdot (y-y_0)^j+\varepsilon^{N+1}\nu_{N+1},\\
\label{sigma b 2}
b(t,\varepsilon,y)&=\sum_{0\leq i,j}\varepsilon^ib_{ij}(t)\cdot(y-y_0)^j=\sum_{i+j\leq N}\varepsilon^ib_{ij}(t)\cdot (y-y_0)^j+\varepsilon^{N+1}v_{N+1}.
\end{align}
Thus, the solution $y^{\varepsilon}(t)$ admit the following Taylor expansion, 
\begin{align}\label{taylor b v}
	dy^{\varepsilon}(t)&=\Big(\sum_{i+j\leq N}\varepsilon^ib_{ij}(t)\cdot (y-y_0)^j+\varepsilon^{N+1}v_{N+1}\Big)dt\nonumber \\
	&+\Big(\sum_{i+j\leq N}\varepsilon^i\cV_{ij}(t)\cdot (y-y_0)^j+\varepsilon^{N+1}\nu_{N+1}\Big)dx(t).
\end{align}
Recall from Definition \ref{taylor expansion}, we have 
\begin{equation}
\label{gk(t)}
y^{\varepsilon}(t)=y_0+\sum_{k=1}^{\infty
}\varepsilon^kg_k(t)=y_0+\sum_{k=1}^{N}\varepsilon^{k}g_k(t)+\varepsilon^{N+1} M_{N+1}(\varepsilon,t).
\end{equation}

The general idea is to match the $\varepsilon$ terms with the same order for the two Taylor expansion \ref{taylor b v} and \ref{gk(t)} as shown above. Then to show $g_j(t)\in\omega(\alpha_j,c_j,\zeta)$, $j=1,\cdots,N$, and $M_{N+1}(\varepsilon,t)\in \omega(\alpha_M,c_M,\zeta)$ is equivalent to show the same properties for the terms appearing in the Taylor expansion of $b(\cdot,t)$ and $\cV(\cdot,t)$. The rest then follows from induction. We break the proof in the following five steps. 

\noindent\textbf{Step 1-matching two Taylor expansion:} Compare the two Taylor expansion with the same power for $\varepsilon$, one has 
\begin{align}
\label{2.1(4)}
dg_1(t)=&(b_{01}(t)g_1(t)+b_{10}(t))dt+(\cV_{01}(t)g_1(t)+\cV_{10}(t))dx(t),\\
\label{2.1(4)'}
dg_{j+1}(t)=&[b_{01}(t)g_{j+1}(t)+K_{j+1}[g_1,\cdots,g_j](t)]dt\nonumber \\
&+[\cV_{01}g_{j+1}+S_{j+1}[g_1,\cdots,g_j](t)]dx(t),
\end{align}
where we denote $S_{j+1}[\cdots](t)$ and $K_{j+1}[\cdots](t)$ as random processes whose values belong to the space of polynomials of $(\mathbb{R}^n)^{\otimes j}$ in $\mathbb{R}^{n\times d}$ and $\mathbb{R}^n$ respectively. By induction, we have 
\begin{align}
S_{1}(t)&=\cV_{10}(t),\quad \text{and}\quad K_{1}(t)=b_{10}(t),\nonumber \\
S_{j+1}[y_1,\cdots, y_j](t)&=\cV_{j+1,0}(t)+\sum_{H(j+1)}\cV_{ir}\cdot [y_{p_1},\cdots,y_{p_r}],\nonumber \\
\label{2.1 (5)}
K_{j+1}[y_1,\cdots, y_j](t)&=b_{j+1,0}(t)+\sum_{H(j+1)}b_{ir}(t)\cdot [y_{p_1},\cdots,y_{p_r}],
\end{align}
where  $H(j+1)$ denote a linear transformation for the integers 
 $(i,r,y_{p_1},\cdots,y_{p_r})$ which satisfies $1+p_1+\cdots+p_r=j+1$ for $0\leq i\leq j$ ,$1\leq r\leq j+1$, $1\leq p_1,\cdots,p_r\leq j$. 
 
 \noindent\textbf{Step 2--transformation of Taylor series:} Given the equations for each term of the Taylor series in \textbf{Step 1}, we now consider linear transformation of the following type:
 \[
 h_j(t)=Q_t g_j(t), \quad j=1,2,\cdots,N,
 \]
 where $Q_t$ denotes the inverse of the Jacobian process, denoted as $J_t$, which satisfies the following equations 
 \begin{align}
\label{K_t}
&dQ_t(y(t))=-Q_t(y(t))[\frac{\partial b}{\partial y_0}(y(t))dt+\frac{\partial \cV}{\partial y_0}(y(t))dx(t)];\\
\label{J_t}
&dJ_t(y(t))=[\frac{\partial b}{\partial y_0}(y(t))dt+\frac{\partial \cV}{\partial y_0}(y(t))dx(t)]J_t(y(t)).
\end{align}
The existence of the Jacobian process $J_t$ and its inverse $Q_t$ follows from \\ \cite{Cass}[Corollary 4.6] in the weak geometric rough path sense satisfying $Q_tJ_t=J_tQ_t=\textbf{Id}$. Next, we consider the process $h_t=Q_ty(t)$. Then follow the product rule, we have $dh_t=Q_tdy(t)+dQ_t y(t)$, which gives us
\begin{align}
\label{h_t}
dh_t=&Q_t[b(y(t))dt+\cV(y(t))dx(t)]-Q_t(y(t))[\frac{\partial b}{\partial y_0}(y(t))dt\nonumber \\
&+\frac{\partial \cV(y(t))}{\partial  y_0}(y(t))dx(t)]\cdot y(t).
\end{align}
Similar to the equations in \eqref{2.1(4)}, for the Taylor expansion terms of $h(t)$, we have 
\begin{equation}
\label{h_j+1}
dh_{j+1}(t)=\tilde K_{j+1}[h_1,\cdots,h_j](t)dt+\tilde S_{j+1}[h_1,\cdots,h_j](t)dx(t),
\end{equation}
where $\tilde S_{j+1}[\cdots](t)$ and $\tilde K_{j+1}[\cdots](t)$ denote random processes whose values belong to the space of polynomial maps.
Following (\ref{2.1(4)}), (\ref{2.1(4)'}), $h_j(t)=Q_tg_j(t)$ and (\ref{h_t}) we then have (by homogeneous property)
\begin{equation}
\label{tilde K_j+1}
\tilde S_{j+1}(t)\circ (Q_t)^{\otimes j}=Q_tS_{j+1}(t),~~~~
\tilde K_{j+1}(t)\circ (Q_t)^{\otimes j}=Q_tK_{j+1}(t).
\end{equation}

\noindent\textbf{Step 3--higher order term transformation:} Recall from our notation convention, we have $dQ_t=-Q_t\cdot [\cV_{01}(t)dx(t)+b_{01}dt]$. For notational convenience, we denote $z_0=0,z_j=\sum_{1\leq k\leq j}\varepsilon^k g_k$ and $y^{\varepsilon}(t)-y^0=z_j+\varepsilon^{j+1}M_{j+1}.$
We then get ,
\begin{align}
dM_{N+1}=&(\cV_{01}M_{N+1}+\Phi_N\cdot M_N+\nu_{N+1})dx(t)\nonumber\\
& +(b_{01}M_{N+1}+\Psi_N\cdot M_N+\upsilon_{N+1})dt,
\end{align}
where $\Phi_N=\Phi_N[\varepsilon,\tilde \cV(t),M_1(t),g_1(t),\cdots,\varepsilon^{N-1}g_{N}(t)]$, and $\Psi_N=\Psi_N[\varepsilon,\tilde b(t),M_1(t)\\,g_1(t),\cdots,\varepsilon^{N-1}g_{N}(t)]$ with $\tilde \cV=[\cV_{ij}(t)]_{i+j\leq N}$ and $\tilde b=[b_{ij}(t)]_{i+j\leq N}$. In particular, $\Phi_N$ and $\Psi_N$ are polynomial maps with constant coefficients from $\mathbb{R}^n$ to $L(\mathbb{R}^d,\mathbb{R}^n)$ and $\mathbb{R}^n$. By applying linear transformation above, similar to  \eqref{h_t}, we define $r_{N+1}(t)$ as below
 \begin{align}
 \label{r_N+1}
 r_{N+1}(t)&=Q_tM_{N+1}(t),
\end{align}
we further differentiate $r_{N+1}$ and denote as 
\begin{align}
	dr_{N+1}=G_Ndt+F_Ndx(t),
\end{align}
which directly implies 
\begin{align}
Q^{-1}F_N&=\Phi_N\cdot M_N+\nu_{N+1},\quad 
Q^{-1}G_N=\Psi_N\cdot M_N +\upsilon_{N+1}.
\end{align}
In particular, $F_N(t)$ and $G_N(t)$ are polynomial functions with constant coefficients of $\varepsilon,~\tilde \cV(t),~\tilde b(t),$
$~M_{1}(t),~g_1(t),~\cdots,
~g_{N-1}(t),~Q_t,~Q^{-1}_t,~\nu_{N+1}(t),~\upsilon_{N+1}(t),~r_N(t)$. 

\noindent\textbf{Step 4--estimate of $g_j$:} Based on the construction from previous three steps. We first show $g_j(t)\in \omega(\alpha_i,c_i,\zeta)$, for $j=1,\cdots, N$. Then we show that $M_{N+1}\in \omega(\alpha_M,c_M,\zeta)$  by using induction. 

Let $U$ be an open set in $\mathbb{R}^n$ where the solution lives on. Denote $K$ as a compact subset of $U$. Denote $T_K^{\varepsilon}$ as the life time of $y^{\varepsilon}(t)$ on a compact set $K$. For a fixed time $T$, take $\zeta=T\wedge T_K^{0}$. Following Proposition \ref{gaussian tail}, the random matrix $Q_t, Q_t^{-1}=J_t$, the solution of the linear equation \eqref{K_t} and \eqref{J_t} with bounded coefficients on [0,$\zeta$], belong to $\omega(\alpha,c,\zeta)$ for some constant $\alpha$ and $c$.
Next, according to \eqref{2.1 (5)} and \eqref{tilde K_j+1}, we know that the random polynomial $\tilde S_{j+1}[\cdots](t)$ and $\tilde K_{j+1}[\cdots](t)$ have coefficients bounded by some constant on $[0,\zeta]$. Thus, according to equation (\ref{h_j+1}), applying recurrence on j and using $\mathbf{(P1)}$ and $\mathbf{(P2)}$, we have $h_j\in \omega(\alpha_i,c_i,\zeta)$. Furthermore, we know that $h_j=Q_tg_j$, which means $g_j$ is a polynomial in the image of $Q^{-1}$, then we conclude $g_j(t)\in \omega(\alpha_i,c_i,\zeta)$ by using $\mathbf{(P1)}$.

\noindent\textbf{Step 5--estimate of $M_{N+1}$:}
 Let $\zeta'=T\wedge T_K^0\wedge T_K^{\varepsilon}$, we already prove that $g_j(t)\in \omega(\alpha_i,c_i,\zeta')$, for $j=1,\cdots,N$. Assume that there exists $\alpha_1, c_1,\cdots,\alpha_N,c_N$ such that $M_j\in \omega(\alpha_i,c_i,\zeta')$, by induction (see details in \cite{Aze}[3-(3)]), we have
\begin{equation}
\label{bound condition}
|\nu_{N+1}|+|\upsilon_{N+1}|\leq c(1+M_1)^{N+1},~~~t\leq \zeta'.
\end{equation}
The boundedness of $\nu_{N+1}$ and $\upsilon_{N+1}$
ensures that the coefficients $F_N$ and $G_N$ of equation (\ref{r_N+1}) are bounded in the Euclidean norm of the polynomial (constant coefficients) in $|Q|,|Q^{-1}|,|M_1|,|M_N|,|g_1|$, $\cdots,|g_{N-1}|$. All these processes belong to the $\omega(\alpha,c,\zeta')$ family, for some $\alpha$'s and $c$'s. By property $\mathbf{(P1)}$, we conclude that $F_N$ and $G_N$ belong to $\omega(\alpha,c,\zeta')$, which implies $r_{N+1}\in \omega(\alpha,c,\zeta')$. Applying $\mathbf{(P1)}$ again, we have $M_{N+1}\in \omega(\alpha,c,\zeta').$
Now it only remains to check the induction from the initial condition that 
\begin{equation}
dM_1(t)=\frac{1}{\varepsilon}[\cV(t,\varepsilon,y^{\varepsilon}(t))-\cV(t,0,y_0)]dy+\frac{1}{\varepsilon}[b(t,\varepsilon,y^{\varepsilon}(t))-b(t,0,y_0)]dt.
\end{equation}
Denote $p_t=(y_0,M_1(t))$ and apply  \eqref{bound condition}, we have
\begin{align}
&dM_1(t)=\nu_{\varepsilon}(t,p_t)dy(t)+\upsilon_{\varepsilon}(t,p_t)dt,\\
&|\nu_{\varepsilon}(t,p_t)|+|\upsilon_{\varepsilon}(t,p_t)|\leq M(1+|M_1(t)|),
\end{align}
where the constant M depends on the compact set K and  the time T. Then following proposition \ref{gaussian tail}, we have $M_1\in \omega(\alpha,c,\zeta')$, for some constants $\alpha,c$.
The proof is thus completed.
\end{proof}

\begin{remark}
The special case check for $t\rightarrow 0$ refers to \cite[Section 5]{Aze}, the only difference is that the linear transformation \eqref{J_t}, \eqref{K_t} in our case does not have the correction term as in \cite[Appendix 7]{Aze}, since our equations are in the Stratonovich form, so it is even easier now.
\end{remark}

\begin{proof}[\textbf{Proof of Lemma \ref{lemma P}}]
Assume that $\tau<\zeta=T_C(r)\wedge T_K^{\varepsilon}\wedge T_K^0\wedge\tilde{T}_K^{\varepsilon}$ from now on. Let K be a compact set of $\mathbb{R}^n$ containing 0, we define $\tilde T_K^{\varepsilon}=\inf\{t;(\varepsilon^{\mathcal O(I)}\Lambda_I)_{\mathcal O(I)\leq N}\\ \notin K\}$ and $T_K^{\varepsilon}$ as before (section 3).
By lemma \ref{lemma M and g}, $g_j\in \omega(\alpha_j,c_j,\zeta)$, for $j=1,\cdots,N$. Observe that $P_{N+1}(\varepsilon,t)$ is a polynomial (with random coefficients) of $\varepsilon$ and $(V_I)_{\mathcal O(I)<N+1}$, and $V_I$ are bounded by some constant according to Theorem \ref{convergence}. Furthermore, $\Lambda_I$ is a linear combination of $Y^I$. Applying $\mathbf{(P2)}$, we know that $Y^I=\int_{\triangle^I}d^{\cV^{-1}I}$ also belongs to $\omega(\alpha,c,\zeta)$. Recall that: 
\begin{align}
\exp \Big(\sum_{I:\mathcal O(I)=1 }^{\infty}\varepsilon^{\mathcal O(I)}\Lambda_I(t)V_I\Big)(y_0)&=\phi((\varepsilon^{\mathcal O(I)}\Lambda_I)_{\mathcal O(I)<N+1})\nonumber\\
 &=y_0+\sum_{k=1}^{N}\varepsilon^k h_k(t)+\varepsilon^{N+1} P_{N+1}(\varepsilon,t).
\end{align}
Note that $P_{N+1}(\varepsilon,t)$ relates to the exit time $\tilde T_K^{\varepsilon}$ and we can choose the compact set $K$ small enough such that $\tilde T_K^{\varepsilon}<T_C(r)$, then we only need to consider $\tilde T_K^{\varepsilon}$ in our proof. Let us fix time $\tau\in(0,\tilde T_K^{\varepsilon})$,\begin{equation}
\label{stopping time estimate}
\mathbb{P}\Big(\sup_{t\in[0,\tau]}\|P_{N+1}(\varepsilon,t)\|\geq \xi\Big)\leq \mathbb{P}\Big( \sup_{t\in[0,\tau]}\|P_{N+1}(\varepsilon,t)\|\geq \xi;\tau \leq \tilde T_K^{\varepsilon}\Big).
\end{equation}
Property $\mathbf{(P1)}$ shows the R.H.S of (\ref{stopping time estimate}) belongs to some $\omega(\alpha',c',\zeta\wedge \tilde T_K^{\varepsilon})$, we conclude that $P_{N+1}\in \omega(\alpha_P,c_P,\zeta\wedge \tilde T_K^{\varepsilon})$ for some $\alpha_P, c_P$.
\end{proof}

\end{appendices}

\bibliographystyle{amsplain}

\end{document}